\documentclass[a4paper,12pt]{article}

\usepackage{amssymb}
\usepackage{amsfonts,dsfont}
\usepackage{amsmath}
\usepackage{amsthm}
\usepackage{cite}
\usepackage{enumitem}
\numberwithin{equation}{section}
\usepackage{xcolor}

\newcommand{\R}{\mathbb{R}}

\newcommand{\N}{\mathbb{N}}

\newcommand{\norm}[1]{\|#1\|}

\newtheorem{definition}{Definition}[section]
\newtheorem{theorem}[definition]{Theorem}
\newtheorem{proposition}[definition]{Proposition}
\newtheorem{lemma}[definition]{Lemma}

\newtheorem{remark}[definition]{Remark}

\renewcommand{\phi}{\varphi}

\newcommand{\cE}{{\mathcal E}}

\newcommand{\cL}{{\mathcal L}}

\newcommand{\cS}{{\mathcal S}}

\newcommand{\bH}{{\mathbb H}}

\allowdisplaybreaks

\begin{document}

\begin{center}
{\bf \large Faber--Krahn inequality for the first Dirichlet  eigenvalues \\[2mm]
of the logarithmic Hardy operator }
\bigskip
\bigskip

{\small
Ying Wang\footnote{yingwang00@126.com}\qquad
Tao Zhang\footnote{zt11227@126.com}\qquad
Wanwan Wang\footnote{wwwang2014@yeah.net, $^{*}$Corresponding author.}$^{,*}$

\bigskip
\medskip

$^{1,2}$\,School of Information Management and Mathematics, Jiangxi University \\
of  Finance and Economics,  Nanchang, Jiangxi 330032,  China\\[10pt]

$^{3}$\,School of Mathematics and Information Science, Nanchang Hangkong University, \\
Jiangxi, Nanchang 330063,  China
}

\medskip

\begin{abstract}\small
The purpose of this paper is to study Faber--Krahn inequality  of
 Dirichlet eigenvalues $\{ \lambda^{\ln,\mu}_{k}(\Omega)\}_{k\in \N}$
 for the logarithmic Hardy operator
\[
\mathcal{L}_{\mu}=(-\Delta)^{\ln}+2\mu\log\frac{1}{|x|}
\]
on a bounded Lipschitz domain $\Omega\subset\mathbb{R}^{N}$ containing the origin, with $\mu>-1$.
We  establish a Faber--Krahn inequality
for $\mu\in(-1,0)$, the  ball centered at the origin minimizes the first eigenvalue at fixed volume,
with equality only for that ball. Finally, we show this inequality fails
for $\mu>0$ via an explicit two-ball-and-tube construction.

\end{abstract}
\end{center}
{\footnotesize

 \textbf{Keywords:}  Logarithmic  Hardy operator; Eigenvalue; Faber--Krahn inequality.

 \textbf{MSC2010:} 35P15;  35R09; 47G30.
}
\bigskip

\setcounter{equation}{0}

\section{Introduction}
Let $\Omega\subset \R^N$ be a  bounded Lipschitz domain  containing the origin with $N\geq 1$, $\mu> -1$
 and $\cL_\mu$  the logarithmic Hardy operator defined by
 $$\cL_\mu u=(-\Delta)^{\ln} u+2\mu\big (\log\frac1{|x|}\big) u,$$
where the operator $(-\Delta)^{\ln}$ is the logarithmic Laplacian with Fourier symbol $2\ln|\xi|$  and it
has  the expression
 \begin{align}\label{ln op-1}
  (-\Delta)^{\ln} u(x)&=c_N\,  {\rm p.v.} \int_{B_1(x)}\frac{u(x)-u(y)}{|x-y|^{N}}\,dy	
  - c_N\int_{\mathbb{R}^N\setminus B_1(0)}\frac{u(x+y)}{|y|^N}\,dy+\rho_N\,u(x),
\end{align}
for $u\in C_c^\beta(\R^N)$ and some $\beta>0$,
here $B_1(x)$ denotes the open unit ball in $\mathbb{R}^N$ centred at  $x$,
$$
	c_N=\pi^{-\frac N2}\Gamma(\frac N2),
	\qquad \
	\rho_N=2\ln 2+\psi\big(\frac N2\big)-\gamma,
$$
$\Gamma$ denotes the Gamma function,	 $\gamma=-\Gamma'(1)$ is the Euler--Mascheroni constant,
$\psi=\frac{\Gamma'}{\Gamma}$ is the Digamma function.

Over the past several decades, elliptic and parabolic equations involving the fractional Laplacian have been extensively studied,
and numerous novel nonlocal operators have emerged due to applications in diverse scientific and mathematical disciplines,
 see \cite{BV16,PK18,MR16,S24} and reference therein.  In   \cite{ChenWeth2019},  a new nonlocal operator named the logarithmic Laplacian $(-\Delta)^{\ln}$, defined in (\ref{ln op-1}),  is derived  by considering  the limit of the fractional Laplacian $(-\Delta)^s$ as $s\to0^+$. More precisely, for $u\in C^\beta_c(\mathbb{R}^n)$ with some $\beta>0$, the function $(-\Delta)^{\ln}u$ is uniquely defined by
	\begin{equation}\label{eq:def-log-Laplace}
		(-\Delta)^s u(x)=u(x)+ s\,(-\Delta)^{\ln}u(x)+o(s)\quad\text{in}\quad L^p(\mathbb{R}^n),\ 1<p\le \infty\quad \text{as }s\to 0^{+},
	\end{equation}
	or, equivalently,
	\[
	(-\Delta)^{\ln}u=\big.\frac{d}{ds}(-\Delta)^s u\big|_{s=0}\quad\text{in}\quad L^p(\mathbb{R}^n),\ 1<p\le \infty.
	\]
Subsequently, the logarithmic Laplacian has emerged as a prominent object of study due to its extremal properties.
 \cite{CHW23} establishes that the operator $(-\Delta)^{\ln}$ admits an extension property  fundamentally distinct from the Caffarelli-Silvestre extension associated with the fractional Laplacian (cf.~\cite{CS07}). Motivated particularly by order-dependent optimization problems involving small optimal orders, \cite{FJ22,AS23,HS22,DN21} study
  Dirichlet problems for $(-\Delta)^s$ in the asymptotic regime $s \to 0^+$.
  Applications of such problems arise, for instance, in image processing and population dynamics \cite{AB17,SV17,PV18}. Most recently, the logarithmic Laplacian has been extended to arbitrary complete Riemannian manifolds in \cite{C25,FS25}.
Particularly, qualitative properties for the Dirichlet eigenvalues have been investigated in \cite{LW,ChenVeron2023} and the references therein.

Recently,  the authors in \cite{LWY} have found  a sequence of  the eigenvalues
 $\{ \lambda^{\ln,\mu}_{k}(\Omega)\}_{k\in \N}$
of problem
\begin{equation}\label{eq: 1.1}
 \left\{
\begin{array}{lll}
\displaystyle  \cL_\mu u= \lambda u\quad
 & {\rm in}\  \   \Omega,
\\[2mm]
 \phantom{ -\ \,   }
 \displaystyle u=0 \quad
 & {\rm in}\  \, \R^N\setminus  \Omega,
 \end{array}
 \right.
 \end{equation}
satisfying
\[
-\infty< \lambda^{\ln,\mu}_{1}(\Omega)< \lambda^{\ln,\mu}_{2}(\Omega)\le \cdots\le \lambda^{\ln,\mu}_{k}(\Omega)\leq \cdots
\]
with corresponding eigenfunctions $\phi^{\ln,\mu}_{k}$, $k\in \N$ and
$\displaystyle \lim_{k\to \infty}\lambda^{\ln,\mu}_{k}(\Omega) = +\infty$.
 Our interest of this paper is to continue to study the properties of
 eigenvalues $\{ \lambda^{\ln,\mu}_{k}(\Omega)\}_{k\in \N}$.\smallskip

The first property is the monotonicity of the eigenvalues with respect to the domain.

 \begin{proposition}\label{pr 2-ineq}
Let \(N\ge 1\), \(\mu>-1\) and  \(\Omega_1,\Omega_2\subset\R^N\) be bounded
Lipschitz domains such that
\[
0\in \Omega_1\subset \Omega_2.
\]
For $i=1,2$, $\{ \lambda^{\ln,\mu}_{k}(\Omega_i)\}_{k\in \N}$ denote the sequence of  eigenvalues for problem
(\ref{eq: 1.1}) with $\Omega=\Omega_i$.

Then  for every \(k\in\N\),
\[
  \lambda_k^{\ln,\mu}(\Omega_1)\ge \lambda_k^{\ln,\mu}(\Omega_2).
\]
\end{proposition}

We next consider the Faber--Krahn inequality for the first Dirichlet eigenvalue  $\lambda^{\ln,\mu}_{1}$
of the operator $\cL_\mu$. It is well known that the Faber-Krahn inequality,
as one of  the most fundamental properties of the first Dirichlet eigenvalue  $\lambda^L_{1}$
of a differential operator $L$, states that, for any bounded domain
$\Omega \subset \mathbb{R}^N$ and any ball $B \subset \mathbb{R}^N$ with $|B| = |\Omega|$,
\[
\lambda^L_1(\Omega) \geq \lambda^L_1(B),
\]
with equality if and only if $\Omega$ is a ball (up to translation and null sets).
Equivalently, the scale-invariant quantity $|\Omega|^{\frac2N}\lambda^L_1(\Omega)$ satisfies
\[
|\Omega|^{\frac2N}\lambda^L_1(\Omega) \geq |B|^{\frac2N}\lambda^L_1(B),
\]
whose right-hand side depends only on dimension $N$. This inequality was first conjectured by
Faber \cite{Faber1923} and rigorously established by Krahn \cite{Krahn1925}.
Brasco--Philippis--Velichkov \cite{BPV2015} studied
the deficit $\lambda^L_1(\Omega) - \lambda^L_1(B)$ in terms of the Fraenkel asymmetry of $\Omega$.
 Extensions to the fractional Laplacian have also been developed, see \cite{F18-1,F20-1} and references therein.

The Faber--Krahn inequality of the first Dirichlet eigenvalue  $\lambda^{\ln,\mu}_{1}$
of  $\cL_\mu$ is stated as follows.

\begin{theorem}\label{teo 1.2}
Let  $N\geq 1$, $B_R(0)$ be the ball centered at the origin and $\lambda^{\ln,\mu}_{1}$
denote the first Dirichlet eigenvalue of $\cL_\mu$.

$(i)$ If $\mu\in(-1,0]$, then for any bounded Lipschitz domain $\Omega\subset\R^N$
with $0\in\Omega$ and $|\Omega|=|B_R(0)|$,
\begin{equation}\label{eq:FK-small}
  \lambda_1^{\ln,\mu}(\Omega)
  \;\geq\;
  \lambda_1^{\ln,\mu}(B_R(0)).
\end{equation}
Equality holds if and only if $\Omega=B_R(0)$.

$(ii)$ If $\mu>0$, then there exists a bounded Lipschitz domain $\Omega$ such that $0\in\Omega$,
$|\Omega|=|B_R(0)|$ and
\begin{equation}\label{eq:FK-fail}
  \lambda_1^{\ln,\mu}(\Omega)
  \;<\;
  \lambda_1^{\ln,\mu}(B_R(0)).
\end{equation}
\end{theorem}

\smallskip

When $\mu=0$, for the pure logarithmic Laplacian
$\cL_0=(-\Delta)^{\ln}$, the Faber--Krahn inequality
\eqref{eq:FK-small} holds for all volumes, as established by
Chen--Weth \cite{ChenWeth2019}
via the relationship
\[
\lambda_1^{\ln}(\Omega)
=
\lim_{s\to0^+}
\frac{\lambda_{1,s}(\Omega)-1}{s},
\]
together with the fractional Faber--Krahn inequality, where $\lambda_1^{\ln}$ is the first Dirichlet eigenvalue of
$(-\Delta)^{\ln}$, $\lambda_{1,s}$ is the first Dirichlet eigenvalue of $(-\Delta)^s$.
While   the rigidity of  Faber--Krahn inequality for $\mu=0$
is built by Chen--Chen--Hauer \cite{CCH26} recently.  This phenomena  is
in sharp contrast with the fractional logarithmic Laplacian \((-\Delta)^{s+\ln}\),
whose Fourier symbol is \(|\xi|^{2s}\ln|\xi|^2\) and for which the
Faber--Krahn inequality holds for small volumes but fails for
sufficiently large volumes when \(N\geq 2\), there exist anisotropic boxes
\(Q\) with \(|Q|=|B_R(0)|\) such that
\(\lambda_1^{s+\ln}(Q)<\lambda_1^{s+\ln}(B_R(0))\).
Theorem~\ref{teo 1.2} $(i)$ shows that the same
conclusion persists for every $\mu\in(-1,0)$, without any small-volume
restriction.  However,  for $\mu\in(-1,0)$, the minimizing ball must be centered at the origin,
while, it could be invariant up to the translatation when $\mu=0$.

 We would like to mention  that   that the operator
\((-\Delta)^{s+\ln}\) is a Fourier multiplier, and its symbol
\(|\xi|^{2s}\ln|\xi|^2\) attains its minimum on a sphere of codimension
\(1\), which allows anisotropic domains to concentrate the Fourier mass
near that sphere and beat the ball at large volumes. For the logarithmic
Hardy operator with \(\mu\in(-1,0)\) no such phenomenon occurs, since $\cL_\mu$
is not a Fourier multiplier for $\mu\neq0$
(see Proposition~\ref{lem:fourier-hardy}), and the potential
\(W_\mu(x)=2\mu\log\frac1{|x|}\) is radially increasing, so the
symmetric decreasing rearrangement decreases the potential energy.
Combined with the P\'olya--Szeg\H{o} inequality for \((-\Delta)^{\ln}\),
the full quadratic form decreases under symmetrization, and the ball
centered at the origin is optimal at every volume.

Note that the first eigenvalue still becomes negative for large domains. In fact, by the scaling property
\(\lambda_1^{\ln,\mu}(r\Omega)=\lambda_1^{\ln,\mu}(\Omega)-2(1+\mu)\log r\)
with \(1+\mu>0\), one has \(\lambda_1^{\ln,\mu}(r\Omega)\to-\infty\) as \(r\to+\infty\)
and \(\lambda_1^{\ln,\mu}(r\Omega)\to+\infty\) as \(r\to0^+\). What the
Faber--Krahn inequality controls is the  difference
\(\lambda_1^{\ln,\mu}(\Omega)-\lambda_1^{\ln,\mu}(B_R(0))\), which is
dilation-invariant by the same scaling property and is always nonnegative by Theorem~\ref{teo 1.2}.

In contrast,    Theorem~\ref{teo 1.2} $(ii)$ shows that for $\mu>0$, the
Faber--Krahn inequality \eqref{eq:FK-small} fails. In this case we just construct a particular domain
whose first eigenvalue is less than the one in a ball centered at origin with the same volume.
Here we construct an explicit two-ball-and-tube domain satisfying  (\ref{eq:FK-fail}).


\smallskip

The paper is organized as follows.  Section~2 is devoted to give the proof of Proposition \ref{pr 2-ineq} and shows that  $\mathcal{L}_\mu$ is not a Fourier multiplier for $\mu\neq0$. Section~3  proves the Faber--Krahn inequality
 for $\mu\in(-1,0)$ and its failure for $\mu>0$ in Theorem~\ref{teo 1.2}.

\section{Preliminary}
\subsection{Basic properties}

Let $C_c^\infty(\Omega)$ be the functions in $C^\infty(\R^N)$ with the support in $\Omega$ and
 $$\bH^{\ln}_0 (\Omega):= \overline{C_c^\infty(\Omega)}^{\|\cdot\|_{\ln }},$$
 where
$$\|u\|_{\ln}=\Big(c_N\int\!\!\!\!\int_{\{x,y\in\R^N: \, |x-y|<1\}}\frac{\big(u(x)-u(y)\big)^2}{|x-y|^{N}}\,dxdy +\int_{\Omega} u(x)^2\,dx\Big)^{\frac12}  \, .$$
Note that $\bH^{\ln}_0 (\Omega)$ is a Hilbert space with the inner product
$$ \langle u,v\rangle_{\ln}=c_N
\int\!\!\!\!\int_{\{x,y\in\R^N: \, |x-y|<1\}}\frac{\big(u(x)-u(y)\big)\big(v(x)-v(y)\big) }{|x-y|^{N}}\, dxdy
+\int_{\Omega} u(x)v(x)\,dx. $$

We first recall the following Hardy type inequality.
\begin{proposition}\cite[Proposition 1.1]{LWY}\label{lem:domain-hardy}
Let  $  \Omega\subset \R^N$ be a  bounded Lipschitz domain with $N\geq 1$ containing the origin.
For every $u\in C_c^\infty(\Omega)$, then 
\begin{equation}\label{eq 01}
2\int_\Omega\big(\log\frac1{|x|}\big)\,u(x)^2\,dx+ d_{N,\Omega}   \norm{u}_{2}^2\leq \|u\|_{\ln}^2\,,
\end{equation}
where
$$d_{N,\Omega}= 2\psi\!\big(\frac N4 \big)-2\psi\!\big(\frac N2 \big) +1-\gamma-C_N |\Omega|^{\frac12},\qquad \ C_N=c_N(\int_{\R^N\setminus B_1(0)} |y|^{-2N}dy)^{\frac12}\, .$$
\end{proposition}

For $\mu>-1$ and $u,v\in C^{\infty}_c(\Omega)$, let us denote
\begin{equation}\label{eq 02}
\langle u,v\rangle_{\ln,\mu} = \langle u,v\rangle_{\ln} +2\mu \int_{\Omega}  \big(\log\frac1{|x|} \big)_+u(x) v(x)\,dx+ (d_{N,\Omega})_+ \int_{\Omega}  u(x) v(x)\,dx,
\end{equation}
where  $a_\pm=\max\{0,\pm a\}$. The corresponding induced norm
\begin{equation}\label{eq 03}
\|u\|_{\ln,\mu}=\sqrt{ \langle u,u\rangle_{\ln,\mu} }\,.
\end{equation}
By the Hardy type inequality (\ref{eq 01}), it shows that
\[
\bH_0^{\ln,\mu}(\Omega)
:=
\overline{C_c^\infty(\Omega)}^{\|\cdot\|_{\ln,\mu}}
\]
 is a Hilbert space with the inner product (\ref{eq 02}).

\begin{proposition}\cite[Theorem 1.2]{LWY}
	\label{prop:energy-space}
	  Let  $  \Omega\subset \R^N$ be a  bounded Lipschitz domain with $N\geq 1$ containing the origin.
Then
		\begin{equation}\label{eq:space-identification}
			\bH_0^{\ln,\mu}(\Omega)\cong \bH_0^{\ln}(\Omega)\qquad {\rm for}\ \ \mu>-1,
			\end{equation}
	and the embedding
		\begin{equation}\label{eq:perturbed-compact}
			\mathbb H_0^{\ln,\mu}(\Omega)\hookrightarrow L^2(\Omega)
		\end{equation}
		is continuous for $\mu\geq -1$
		and compact for $\mu>-1$ and non-compact for $\mu=-1$.
\end{proposition}

Now we are position to show the proof of Proposition \ref{pr 2-ineq}.

\medskip

 \noindent{\bf Proof of Proposition \ref{pr 2-ineq}.}
 We use the variational min--max characterization of the Dirichlet
eigenvalues. Denote
\begin{align}\label{Euv1}
  \cE_{\ln,\mu,\Omega}(u,u)
    =\int_{\Omega}u(x)(-\Delta)^{\ln}u(x)\,dx
    +2\mu \int_{\Omega} \Big(\log\frac1{|x|}\Big) u(x)^2 \,dx.
\end{align}
By Proposition  \ref{prop:energy-space},  the embedding
\[
\bH_0^{\ln,\mu}(\Omega)\hookrightarrow L^2(\Omega)
\]
is compact for \(\mu>-1\).
Then the eigenvalues admit the min--max characterization
\[
\lambda_k^{\ln,\mu}(\Omega)
=
\min_{\substack{H\subset \bH_0^{\ln,\mu}(\Omega)\\ \dim H=k}}
\max_{u\in H\setminus\{0\}}
\frac{\cE_{\ln,\mu,\Omega}(u,u)}
{\|u\|_{L^2(\Omega)}^2},
\qquad k\in\N,
\]
with eigenvalues counted with multiplicity.

Let
\[
R_\Omega(u)=\frac{\cE_{\ln,\mu,\Omega}(u,u)}{\|u\|_{L^2(\Omega)}^2}.
\]
Then we have that
\[
\lambda_k^{\ln,\mu}(\Omega)
=
\min_{\substack{H\subset \bH_0^{\ln,\mu}(\Omega)\\ \dim H=k}}
\max_{u\in H\setminus\{0\}} R_\Omega(u).
\]

Since \(\Omega_1\subset\Omega_2\), for \(u\in C_c^\infty(\Omega_1)\), we observe that   
\(u\in C_c^\infty(\Omega_2)\) and 
$$
\|u\|_{L^2(\Omega_2)}^2 =\|u\|_{L^2(\Omega_1)}^2,\qquad \ 
\cE_{\ln,\mu,\Omega_2}(u,u)=
\cE_{\ln,\mu,\Omega_1}(u,u).
$$
Then it yields that
\[
R_{\Omega_2}(u)=R_{\Omega_1}(u).
\]
By density of \(C_c^\infty(\Omega_1)\) in \(\bH_0^{\ln,\mu}(\Omega_1)\),  
we note that if \(H\subset \bH_0^{\ln,\mu}(\Omega_1)\) is a  \(k\)-dimensional
subspace, then \( H\subset \bH_0^{\ln,\mu}(\Omega_2)\) is also \(k\)-dimensional and
\[
\max_{u\in H\setminus\{0\}} R_{\Omega_1}(u)
=
\max_{v\in H\setminus\{0\}} R_{\Omega_2}(v).
\]
Using the min--max formula for \(\Omega_2\), we obtain that
\begin{align*}
\lambda_k^{\ln,\mu}(\Omega_2)
 &=
\min_{\substack{W\subset \bH_0^{\ln,\mu}(\Omega_2)\\ \dim W=k}}
\max_{v\in W\setminus\{0\}} R_{\Omega_2}(v)
\\&\le
\min_{\substack{H\subset \bH_0^{\ln,\mu}(\Omega_1)\\ \dim H=k}}\,
\max_{u\in H\setminus\{0\}} R_{\Omega_1}(u)
=\lambda_k^{\ln,\mu}(\Omega_1).
\end{align*}
This completes the proof.\hfill$\Box$\medskip

\subsection{Rearrangement}

\begin{definition}
Let \(N\ge 1\) and \(f:\R^N\to\R\) be a measurable function such that
\[
  d_f(t):=\bigl|\{x\in\R^N: |f(x)|>t\}\bigr|<+\infty
  \qquad\text{for every} \ t>0.
\]
Denote \(\omega_N=|B_1(0)|\). The  symmetric decreasing rearrangement
of \(f\) is the function \(f^\#:\R^N\to[0,\, +\infty]\) defined by
\[
  f^\#(x)= \inf\Bigl\{t\ge0:\ d_f(t)\le \omega_N |x|^N\Bigr\}, \qquad x\in\R^N.
\]
Equivalently, for every \(t\ge0\),
\[
  \{x\in\R^N: f^\#(x)>t\}
  =
  B_{r(t)}(0),
  \qquad
  |B_{r(t)}(0)|=d_f(t),
\]
where \(B_{r(t)}(0)\) is the open ball centered at the origin with the same
measure as the superlevel set \(\{x:|f(x)|>t\}\).
\end{definition}

Observe that  for
\(1\le p\le+\infty\) and $f\in L^p(\R^N)$, it holds that
\[
  \|f^\#\|_{L^p(\R^N)}=\|f\|_{L^p(\R^N)}.
\]

\begin{lemma}\label{lm 2.1}
Let \(N\ge1\),   \(\Omega\subset\R^N\) be a bounded Lipschitz domain containing the origin, 
\(R>0\) satisfy that \(|B_R(0)|=|\Omega|\).  For 
 \(u\in \bH_0^{\ln}(\Omega)\), denote by \(u^\#\) the symmetric decreasing rearrangement.

Then
\[
  u^\#\in \bH_0^{\ln}(B_R(0)).
\]
\end{lemma}
\noindent{\bf Proof.} Let \(u\in C_c^{\infty}(\Omega)\),  we have that
\[
  \cE_{\ln,0,\Omega}(u,u)<+\infty.
\]
For \(s\in(0,1)\), the classical P\'olya--Szeg\H{o} inequality for the
fractional Laplacian shows that
\[
  \norm{(-\Delta)^{s/2}u^\#}_{L^2(\R^N)}^2
  \le
  \norm{(-\Delta)^{s/2}u}_{L^2(\R^N)}^2
\]
and then
\[
 \int_{\R^N}|\xi|^{2s}|\widehat{u^\#}(\xi)|^2\,d\xi  \le
 \int_{\R^N}|\xi|^{2s}|\widehat u(\xi)|^2\,d\xi
  <+\infty,
\]
where \[
  \widehat u(\xi)=(2\pi)^{-\frac N2}\int_{\R^N}e^{-{\bf i}  x\cdot\xi}u(x)\,dx.
\]
Thus, we have that
\[
\int_{\R^N}|\xi|^{2s}|\widehat{u^\#}(\xi)|^2\,d\xi- \norm{u^\#}_{L^2(\R^N)}^2
  \le
 \int_{\R^N}|\xi|^{2s}|\widehat u(\xi)|^2\,d\xi-\norm{u}_{L^2(\R^N)}^2,
\]
since \[
 \norm{u^\#}_{L^2(\R^N)}^2= \norm{u^\#}_{L^2(B_R(0))}^2 =
  \norm{u}_{L^2(\Omega)}^2=\norm{u}_{L^2(\R^N)}^2. \]
 By the fact that $$|\xi|^{2s}=1+2s\log|\xi|+O(s^2) \qquad {\rm for} \ \ \xi\not=0$$ and 
passing to the limit as \(s\to0^+\), we obtain
\begin{equation}\label{eq 2.1-re}
  \cE_{\ln,0,B_R(0)}(u^\#,u^\#)\le
  \cE_{\ln,0,\Omega}(u,u) <+\infty.
\end{equation}
Therefore, it holds that
\begin{equation}\label{eq 04}
  u^\#\in \bH_0^{\ln}(B_R(0)).
\end{equation}

Since \(C_c^\infty(\Omega)\) is density in \(\bH_0^{\ln}(\Omega)\)  (see \cite{ChenWeth2019}), 
then (\ref{eq 04}) holds for any  \(u\in \bH_0^{\ln}(\Omega)\)
via taking a sequence functions in \(C_c^\infty(\Omega)\) approaching to 
$u$ in $\bH_0^{\ln}(\Omega)$ if necessary.\hfill$\Box$\medskip

\begin{lemma}\label{lm 2.2}
Assume that \(N\ge 1\), \(w:[0,\infty)\to\mathbb R\) is strictly increasing
 and \(f:[0,\infty)\to[0,\infty)\) is strictly increasing, locally Lipschitz continuous. 
 Let \(u\) be a nonnegative function such that \(u\in L^1(\R^N,w\,dx)\), then 
\begin{equation}\label{inq 2.1}
 \int_{\mathbb R^N} w(|x|)\, f(u^\#(x))\,dx
   \;\le\;
  \int_{\mathbb R^N} w(|x|)\, f(u(x))\,dx,
\end{equation}
furthermore, the equality holds if and only if
\[
  u^\#=u \qquad \text{a.e. in } \mathbb R^N.
\]
\end{lemma}

 \noindent{\bf Proof.} 
When $f(u)=u$, this lemma is a standard consequence of the Hardy-Littlewood rearrangement inequality,
see \cite[Theorem 3.4]{Lieb1} or \cite[Theorem 368]{Hardy1}.  For the general increasing $f$, we provide the detailed  proof.

Set \(v:=f(u)\). Since \(f\) is strictly increasing and locally Lipschitz continuous, the superlevel sets of
\(v\) and \(u\) are related by
\[
  \{x\in\mathbb R^N: v(x)>t\}
  =
  \{x\in\mathbb R^N: u(x)>f^{-1}(t)\}
\]
for \(t\) in the range of \(f\). Hence
\[
  |\{x: v(x)>t\}|
  =
  |\{x: u(x)>f^{-1}(t)\}|.
\]
By definition of the symmetric decreasing rearrangement,
\[
  |\{x: u^\#(x)>s\}|=|\{x: u(x)>s\}|
  \qquad \text{for all } s\ge 0.
\]
Therefore,
\begin{align*}
  |\{x: f(u^\#(x))>t\}|
  &=
  |\{x: u^\#(x)>f^{-1}(t)\}|
  \\[1mm]&\ \, =
  |\{x: u(x)>f^{-1}(t)\}|
 \, =
  |\{x: v(x)>t\}|.
\end{align*}
Since \(f\) is increasing and \(u^\#\) is radially decreasing, the function
\(f(u^\#)\) is also radially decreasing. Thus \(f(u^\#)\) is the symmetric
decreasing rearrangement of \(v=f(u)\), i.e.
\[
  (f(u))^\# = f(u^\#).
\]

Now apply the Hardy--Littlewood rearrangement inequality to \(v=f(u)\).
For every \(t\ge 0\), let \({B_{r(t)}(0)}\subset\mathbb R^N\) be the ball centered at the
origin with
\[
  |{B_{r(t)}(0)}|=|\{x: v(x)>t\}|.
\]
Since \(w\) is radially increasing, among all sets of a fixed measure the
integral of \(w(|x|)\) is minimized by the centered ball. Hence
\[
  \int_{\{v>t\}} w(|x|)\,dx
  \;\ge\;
  \int_{B_{r(t)}(0)} w(|x|)\,dx.
\]
Using the layer-cake representation,
\[
  \int_{\mathbb R^N} w(|x|) f(u(x))\,dx
  =
  \int_{\mathbb R^N} w(|x|) v(x)\,dx
  =
  \int_0^\infty \int_{\{v>t\}} w(|x|)\,dx\,dt.
\]
Therefore,
\[
\begin{aligned}
  \int_{\mathbb R^N} w(|x|) f(u(x))\,dx
  &=
  \int_0^\infty \int_{\{v>t\}} w(|x|)\,dx\,dt \\[1mm]
  &\ge
  \int_0^\infty \int_{B_{r(t)}(0)} w(|x|)\,dx\,dt \\[1mm]
  &=
  \int_0^\infty \int_{\{v^\#>t\}} w(|x|)\,dx\,dt \\[1mm]
  &=
  \int_{\mathbb R^N} w(|x|) v^\#(x)\,dx \\[1mm]
  &=
  \int_{\mathbb R^N} w(|x|) f(u^\#(x))\,dx,
\end{aligned}
\]
which implies the inequality (\ref{inq 2.1}).

It remains to prove the equality case. Suppose equality holds in the above
chain. Then for almost every \(t\ge 0\),
\[
  \int_{\{v>t\}} w(|x|)\,dx
  =
  \int_{B_{r(t)}(0)} w(|x|)\,dx.
\]
Since \(w\) is strictly increasing, the only sets of a given measure for which
this equality can hold are, up to null sets, the centered balls. Hence
\[
  \{v>t\}={B_{r(t)}(0)}
  \qquad \text{up to a null set, for a.e. } t\ge 0.
\]
Thus
\[
  \mathbf 1_{\{v>t\}}(x)=\mathbf 1_{B_{r(t)}(0)}(x)
  \qquad \text{for a.e. } x\in\mathbb R^N,\ \text{for a.e. } t\ge 0.
\]
Integrating in \(t\), we get
\[
  v(x)
  =
  \int_0^\infty \mathbf 1_{\{v>t\}}(x)\,dt
  =
  \int_0^\infty \mathbf 1_{B_{r(t)}(0)}(x)\,dt
  =
  v^\#(x)
\]
for almost every \(x\in\mathbb R^N\). Hence \(v=v^\#\) a.e., i.e.
\[
  f(u)=f(u^\#)\quad\text{a.e. in } \mathbb R^N.
\]
Since \(f\) is strictly increasing, it is injective, and therefore
\[
  u=u^\#\quad\text{a.e. in } \mathbb R^N.
\]
Conversely, if \(u=u^\#\) a.e., then the two integrals are clearly equal.
The proof ends.
\hfill$\Box$
\medskip

\subsection{Fourier multiplier}
This subsection is devoted to show that $\cL_\mu$ is not a Fourier multiplier when $\mu\neq0$.

\begin{proposition}\label{lem:fourier-hardy}
Let $N\geq 1$ and  $u\in\cS(\R^N)$ be real-valued.
Then 
\begin{equation}\label{eq:fourier-log}
 \widehat{(\log\frac1{|\cdot|})}(\xi)
  =
  \theta_N\,\mathrm{Pf}\frac{1}{|\xi|^N}
+
  \alpha_N\,\delta_0,
\end{equation}
where $ \theta_N=2^{\frac N2-1}\Gamma\bigl(\tfrac N2\bigr)>0$,
$\mathrm{Pf}\frac1{|\xi|^N}$ is the finite-part regularization of the
non-integrable function $|\xi|^{-N}$, defined by
\[
  \Bigl\langle \mathrm{Pf}\frac1{|\xi|^N},\,\varphi\Bigr\rangle
  =
  \int_{|\xi|\ge1}\frac{\varphi(\xi)}{|\xi|^N}\,d\xi
  +
  \int_{|\xi|<1}\frac{\varphi(\xi)-\varphi(0)}{|\xi|^N}\,d\xi,
  \qquad\forall\, \varphi\in\cS(\R^N),
\]
and $\alpha_N$ is a constant depending only on $N$ and on the normalization of
the finite part.


Moreover, 
\begin{equation}\label{eq:hardy-fourier-form}
\begin{aligned}
 \int_{\R^N}\big(\log\frac1{|x|}\big)\,u(x)^2\,dx
  =
  \theta_N
    \int_{|\xi|\ge1}\frac{\widehat{u^2}(\xi)}{|\xi|^N}\,d\xi
    +
    \theta_N\int_{|\xi|<1}\frac{\widehat{u^2}(\xi)-\widehat{u^2}(0)}{|\xi|^N}\,d\xi
  +
\alpha_N \|u\|_{2}^2,
\end{aligned}
\end{equation}
where $\widehat{u^2}$ denotes the Fourier transform of the function $u^2$,
\[
  \widehat{u^2}(\xi)
  =
  (2\pi)^{-\frac {N}2}\, \int_{\R^N}e^{-{\bf i}x\cdot\xi}u(x)^2\,dx
  =
 \bigl(\widehat u*\widehat u\bigr)(\xi).
\]

\end{proposition}

\noindent{\bf Proof.}
For $0<\lambda<N$, the classical Riesz potential formula (see \cite{GS})
gives
\[
  \mathcal F\bigl[|\cdot|^{-\lambda}\bigr](\xi)
  =
  2^{\frac N2-\lambda}
  \frac{\Gamma\bigl(\frac{N-\lambda}{2}\bigr)}{\Gamma\bigl(\frac\lambda2\bigr)}
  \,|\xi|^{\lambda-N}
  =:
  \theta_\lambda\,|\xi|^{\lambda-N}
\]
with
$$ \theta_\lambda=2^{\frac N2-\lambda}
  \frac{\Gamma\bigl(\frac{N-\lambda}{2}\bigr)}{\Gamma\bigl(\frac\lambda2\bigr)}.$$
Since
$|x|^{-\lambda}=1-\lambda\log|x|+O(\lambda^2)$ as $\lambda\to0^+$, we have that
\[
  \mathcal F\bigl[|\cdot|^{-\lambda}\bigr]
  =
  (2\pi)^{\frac N2}\delta_0-\lambda\,\widehat{\log|\cdot|}+O(\lambda^2).
\]
Since $\Gamma(\lambda/2)=2/\lambda+\Gamma'(1)+O(\lambda)$ as $\lambda\to0^+$, one has
$\theta_\lambda=\lambda \theta_N+\tfrac12\lambda^2 \theta_N''+O(\lambda^3)$, where $\theta_N''=\frac{d}{d\lambda} \theta_\lambda\big|_{\lambda=0}$, while for $\varphi\in\cS(\R^N)$,
\[
  \bigl\langle|\xi|^{\lambda-N},\varphi\bigr\rangle
  =
  \frac{|\mathbb S^{N-1}|}{\lambda}\,\varphi(0)
  +
  \Bigl\langle \mathrm{Pf}\frac1{|\xi|^N},\varphi\Bigr\rangle
  +
  o(1)
  \quad{\rm as}\ \,  \lambda\to0^+.
\]
By the fact that $\theta_N|\mathbb S^{N-1}|=(2\pi)^{\frac N2}$, it yields
\[
  \widehat{(\log\frac1{|\cdot|})}(\xi)
  =
  \theta_N\,\mathrm{Pf}\frac1{|\xi|^N}
  +
 \alpha_N\,\delta_0
  \quad {\rm with}
\ \  \alpha_N=\frac{\theta_N''\,|\mathbb S^{N-1}|}{2},
\]
and then \eqref{eq:fourier-log} holds.

\smallskip

For $u\in\cS(\R^N)$, one has $u^2\in\cS(\R^N)$, and by Parseval's identity, we have that
\[
  \int_{\R^N}(\log\frac1{|x|})\,u(x)^2\,dx
  =
  \Bigl\langle\log\frac1{|\cdot|},\,u^2\Bigr\rangle
  =
  \Bigl\langle
   \widehat{\log\frac1{|\cdot|}},\,\widehat{u^2}
  \Bigr\rangle .
\]
Combining with \eqref{eq:fourier-log} and the fact that
$$\widehat{u^2}=\widehat u*\widehat u,\qquad
\widehat{u^2}(0)=\int_{\R^N}u(x)^2dx=\|u\|_{2}^2,$$
we obtain \eqref{eq:hardy-fourier-form}.
\hfill$\Box$

\section{Faber--Krahn Inequality}\label{sec4}

\subsection{The case of $\mu\in(-1,0]$}

This subsection is devoted to prove the Faber--Krahn inequality for the first
Dirichlet eigenvalue of the logarithmic Hardy operator $ \cL_\mu$
in the regime $\mu\in(-1,0]$.


\medskip

\noindent{\bf Proof of Theorem \ref{teo 1.2}$(i)$.}  When $\mu=0$, the Faber--Krahn inequality
has been established in  \cite{ChenWeth2019} and  the rigidity  has been done  in \cite{CCH26}. So in the sequel, we only
deal with the case $\mu\in(-1,0)$.

Recall that 
\begin{align}\label{eq:energy-decomp01a}
  \cE_{\ln,\mu,\Omega}(u,u)
    =\int_{\Omega}u(x)(-\Delta)^{\ln}u(x)\,dx
    +2\mu \int_{\Omega} \Big(\log\frac1{|x|}\Big) u(x)^2 \,dx,
\end{align}
by \cite[Theorem 1.3]{LWY}, for $\mu>-1$,  we have that the first Dirichlet eigenvalue of 
the logarithmic Hardy operator $ \cL_\mu$ could be given by
\[
  \lambda_1^{\ln,\mu}(\Omega)
  =
  \min\Bigl\{
    \cE_{\ln,\mu,\Omega}(u,u):
    u\in\bH_0^{\ln,\mu}(\Omega),\ \norm{u}_{L^2(\Omega)}=1
  \Bigr\},
\]
and the minimum is attained by a positive eigenfunction
$\phi_1^{\ln,\mu}\in\bH_0^{\ln,\mu}(\Omega)$. By Proposition \ref{prop:energy-space}, we note that 
$\phi_1^{\ln,\mu}\in\bH_0^{\ln}(\Omega)$ when $\mu>-1$.

 \smallskip

Let $R>0$ satisfy that $|B_R(0)|=|\Omega|$, for $u\in\bH_0^{\ln}(\Omega)$,
denote by $u^\#$  the symmetric decreasing rearrangement,
 then  $\norm{u^\#}_{L^2(\R^N)}=\norm u_{L^2(\R^N)}$ and Lemma \ref{lm 2.1} implies that
$u^\#\in\bH_0^{\ln}(B_R(0))$.

Now we claim that
\begin{equation}\label{eq:symmetrization}
  \cE_{\ln,\mu, B_R(0)}(u^\#,u^\#)
  \;\le\;
  \cE_{\ln,\mu,\Omega}(u,u).
\end{equation}
where the equality holds if and only if $u^\#=u$ a.e. in $\R^N$. 
Indeed,  since  $\mu\in(-1,0)$,  then  $2\mu\log\frac1{|x|}$ is radially symmetric and  is strictly increasing about $|x|$
on $(0,\infty)$, then we apply Lemma \ref{lm 2.2} with $f(u)=u^2$ to obtain that
 $$2\mu \int_{B_R(0)}\log\frac1{|x|} u^\#(x)^2\,dx\leq 2\mu \int_{\Omega}\log\frac1{|x|} \,u(x)^2\,dx,  $$
which, combining (\ref{eq 2.1-re}), i.e.
$$\cE_{\ln,0, B_R(0)}(u^\#,u^\#)
  \;\le\;
  \cE_{\ln,0,\Omega}(u,u)$$
  implies that
 \eqref{eq:symmetrization}. \smallskip

 Therefore, we have that
\[
  \cE_{\ln,\mu,\Omega}(u,u)
  \;\ge\;
  \cE_{\ln,\mu, B_R(0)}(u^\#,u^\#)
  \;\ge\;
  \lambda_1^{\ln,\mu}\bigl(B_R(0)\bigr),
\]
since $u^\#$ is  admissible  for the ball $B_R(0)$ with
$\norm{u^\#}_{L^2(B_R(0))}=1$. Taking the infimum over all $u\in\bH_0^{\ln}(\Omega)$
with $\norm u_{L^2(\Omega)}=1$ yields
\[
  \lambda_1^{\ln,\mu}(\Omega)
  \;\ge\;
  \lambda_1^{\ln,\mu}\bigl(B_R(0)\bigr).
\]

\smallskip
 Finally, we consider the equality case.
Assume $\lambda_1^{\ln,\mu}(\Omega)=\lambda_1^{\ln,\mu}(B_R(0))$  and let
$\phi_1:=\phi_1^{\ln,\mu}\in\bH_0^{\ln}(\Omega)$ be the positive first
eigenfunction of $\Omega$, normalized by
$\norm{\phi_1}_{L^2(\Omega)}=1$. Then
\begin{align*}
  \lambda_1^{\ln,\mu}(\Omega)
  =
  \cE_{\ln,\mu,\Omega}(\phi_1,\phi_1)
  &\;\ge\;
  \cE_{\ln,\mu, B_R(0)}(\phi_1^\#,\phi_1^\#)
 \\[1mm]& \;\ge\;
  \lambda_1^{\ln,\mu}\bigl(B_R(0)\bigr)
  =
  \lambda_1^{\ln,\mu}(\Omega),
\end{align*}
hence all inequalities above are equalities.
 In particular, equality
holds in \eqref{eq:symmetrization} holds  only if
$\phi_1=\phi_1^\#$ almost everywhere in $\R^N$.

Conversely, equality
is trivial when $\Omega=B_R(0)$. Thus equality holds in
\eqref{eq:FK-small} if and only if  $\Omega=B_R(0)$..
\hfill$\Box$\medskip

\begin{remark}\label{re1}
Under the assumptions of Theorem~\ref{teo 1.2} $(i)$, the ball
$B_R(0)$ minimizes the first eigenvalue among all domains of the same
volume, i.e.
\[
  \lambda_1^{\ln,\mu}(B_R(0))
  =
  \min\Big\{
    \lambda_1^{\ln,\mu}(\Omega):
    \Omega\subset\R^N \text{ bounded Lipschitz},\ 0\in\Omega,\
    |\Omega|=|B_R(0)|
  \Big\},
\]
for every $V=|\Omega|>0$, and the minimum is attained exactly at the ball
centered at the origin.
\end{remark}

\subsection{The case of $\mu>0$}

\noindent{\bf Proof of Theorem \ref{teo 1.2}$(ii)$.}  
Our aim is to construct a   Lipschitz domain $\Omega\subset\R^N$ containing the origin, with
$|\Omega|=|B_R(0)|=V$   such that for  $\mu>0$,
\[
  \lambda_1^{\ln,\mu}(\Omega)
  <
  \lambda_1^{\ln,\mu}(B_R(0)).
\]

To this end,  we fix $V>0$ and let $R>0$ satisfy $|B_R(0)|=V$. Choose $\varepsilon>0$ and
$r>0$ such that
\[
  |B_\varepsilon(0)|+|B_r(0)|<V .
\]
Let $x_0\in\R^N$ with $|x_0|$ large enough which will be chosen later. Connect the two balls
$B_\varepsilon(0)$ and $B_r(x_0)$ by a thin tube $T$ of volume
$V-|B_\varepsilon(0)|-|B_r(0)|$. Denote
\[
  \Omega=B_\varepsilon(0)\cup T\cup B_r(x_0).
\]
Then $\Omega$ is a bounded Lipschitz domain and satisfies that
 $$0\in\Omega,\quad |\Omega|=V=|B_R(0)|,\quad B_r(x_0)\subset\Omega.$$

Since $B_r(x_0)\subset\Omega$, every function
$u\in\bH_0^{\ln}(B_r(x_0))$,  we have that 
$u\in\bH_0^{\ln}(\Omega)$ and  the Rayleigh quotient for $\Omega$ is no
larger than the Rayleigh quotient for $B_r(x_0)$, and hence
\begin{equation}\label{eq:monotonicity-counter}
  \lambda_1^{\ln,\mu}(\Omega)
  \le
  \lambda_1^{\ln,\mu}(B_r(x_0)).
\end{equation}

We now estimate $\lambda_1^{\ln,\mu}(B_r(x_0))$ as $|x_0|\to+\infty$. Recall
that 
\begin{align}\label{1b0}
  \cE_{\ln,\mu, B_r(x_0)}(u,u)
    =\int_{B_r(x_0)}u(x)(-\Delta)^{\ln}u(x)\,dx
    +2\mu \int_{\R^N} \Big(\log\frac1{|x|}\Big) u(x)^2 \,dx,
\end{align}
For $x\in B_r(x_0)$ and $|x_0|>2r$, we have
\[
  \left|\log\frac1{|x|}-\log\frac1{|x_0|}\right|
  =
  \left|\log\frac{|x_0|}{|x|}\right|
  \le
  \log\frac{|x_0|}{|x_0|-r}
  \le
  \frac{2r}{|x_0|}.
\]
Thus, uniformly on $B_r(x_0)$,
\[
  2\mu\log\frac1{|x|}
  =
  2\mu\log\frac1{|x_0|}
  +
  O\!\left(\frac{\mu r}{|x_0|}\right).
\]
Combining with (\ref{1b0}), it yields that  for every $u\in\bH_0^{\ln}(B_r(x_0))$,
\[
  \cE_{\ln,\mu,B_r(x_0)}(u,u)
  =
  \int_{B_r(x_0)}u(x)(-\Delta)^{\ln}u(x)\,dx
  +
  2\mu\log\frac1{|x_0|}\|u\|_{L^2(\R^N)}^2
  +
  O\!\left(\frac{\mu r}{|x_0|}\right)\|u\|_{L^2(\R^N)}^2.
\]
Dividing by $\|u\|_{L^2(\R^N)}^2$ and taking the infimum over all admissible
$u$, we obtain
\begin{equation}\label{eq:far-ball-asymptotic}
  \lambda_1^{\ln,\mu}(B_r(x_0))
  =
  \lambda_1^{\ln,0}(B_r(x_0))
  +
  2\mu\log\frac1{|x_0|}
  +
  O\!\left(\frac{\mu r}{|x_0|}\right).
\end{equation}
Since $\mu>0$, then $2\mu\log\frac1{|x_0|}\to -\infty$ as
$|x_0|\to+\infty$ and
\[
  \lambda_1^{\ln,\mu}(B_r(x_0))
  \to
  -\infty
  \quad\text{as }\, |x_0|\to+\infty.
\]
Then we could choose $|x_0|$  large enough such that
\[
  \lambda_1^{\ln,\mu}(B_r(x_0))
  <
  \lambda_1^{\ln,\mu}(B_R(0)).
\]
Combining this with \eqref{eq:monotonicity-counter}, we get
\[
  \lambda_1^{\ln,\mu}(\Omega)
  \le
  \lambda_1^{\ln,\mu}(B_r(x_0))
  <
  \lambda_1^{\ln,\mu}(B_R(0)).
\]
But $|\Omega|=V=|B_R(0)|$. Hence the Faber--Krahn inequality is
violated in this case of $\mu>0$. The proof ends.\hfill$\Box$

\bigskip

{\small

\noindent {\bf  Conflicts of interest:}
The authors declare that they have no conflicts of interest regarding this work.

\medskip

\noindent {\bf Data availability:} This paper has no associated data.

\medskip

\noindent{{\bf Acknowledgements:} 
Y. Wang is supported by the National Natural Science Foundation of China (No. 12461041) and
by the Jiangxi Provincial Natural Science Foundation (No. 20252BAC240158).
W. Wang is supported by the National Natural Science Foundation of China (Nos. 12661039,
12301252, 12461043), by the Jiangxi Provincial Natural Science Foundation (No. 20242BAB20003) and
by the Doctoral Research Foundation of Nanchang Hangkong University (No. EA202207184).

}

\end{document}